\documentclass[a4paper]{amsart}
\usepackage{amssymb,amsmath}
\usepackage{enumerate}

\usepackage[latin1]{inputenc}
\usepackage{color}
\usepackage[pdftex]{graphicx}

\DeclareMathOperator{\dist}{dist}

\DeclareMathOperator{\graph}{graph}

\DeclareMathOperator{\divergenz}{div}

\def\ol#1{\overline{#1}}
\def\ul#1{\underline{#1}}
\def\R{{\mathbb{R}}}
\def\N{{\mathbb{N}}}

\def\S{{\mathbb{S}}}

\def\D{{\mathcal{D}}}

\def\G{{\mathcal{G}}}

\def\theta{{\vartheta}}
\def\phi{{\varphi}}
\def\epsilon{{\varepsilon}}
\def\dt{{\frac{d}{dt}}}

\def\fracd#1#2{{\frac{\displaystyle #1}{\displaystyle #2}}}

\long\def\umbruch{{\displaybreak[1]}}
\long\def\neueZeile{{\rule{0mm}{1mm}\\[-3.25ex]\rule{0mm}{1mm}}}

\def\heat{\left(\tfrac d{dt}-\Delta\right)}
\def\Amod#1{\left|\nabla^{#1}A\right|^2}
\def\emph#1{\textbf{#1}}

\mathchardef\ordinarycolon\mathcode`\:
  \mathcode`\:=\string"8000
  \begingroup \catcode`\:=\active
    \gdef:{\mathrel{\mathop\ordinarycolon}}
  \endgroup

\newtheorem{theorem}{Theorem}[section]
\newtheorem{lemma}[theorem]{Lemma}
\newtheorem{proposition}[theorem]{Proposition}
\newtheorem{corollary}[theorem]{Corollary}

\newtheorem{definition}[theorem]{Definition}

\newtheorem{remark}[theorem]{Remark}
\newtheorem{assumption}[theorem]{Assumption}

\numberwithin{equation}{section}
\begin{document}
\title{Mean curvature flow without singularities}

\author{Mariel S\'aez}
\thanks{The first author was partially supported by Conicyt under
  grants Fondecyt regular 1110048 and proyecto Anillo ACT-125, CAPDE}
\address{Mariel S\'aez, Departamento de Matem\'aticas, Avda.{}
  Vicu\~{n}a Mackenna 4860. Macul, Santiago, Chile}
\curraddr{}
\def\ChileHome{@mat.puc.cl}
\email{mariel\ChileHome}

\author{Oliver C. Schn\"urer}
\address{Oliver C. Schn\"urer, Fachbereich Mathematik und Statistik,
  Universit\"at Konstanz, 78457 Konstanz, Germany}
\curraddr{}
\def\AmSeeHome{@uni-konstanz.de}
\email{Oliver.Schnuerer\AmSeeHome}
\thanks{}


\date{\today.}

\dedicatory{}


\begin{abstract}
  We study graphical mean curvature flow of complete solutions defined
  on subsets of Euclidean space. We obtain smooth long time
  existence. The projections of the evolving graphs also solve mean
  curvature flow. Hence this approach allows to smoothly flow through
  singularities by studying graphical mean curvature flow with one
  additional dimension. 
\end{abstract}

\maketitle

\tableofcontents

\section{Introduction}

\subsection*{Results} 
We start by stating a simplified version of our main result, which
holds for bounded domains.  Let us consider mean curvature flow for
graphs defined on a relatively open set
\begin{equation}
  \label{set eq}
  \Omega\equiv\bigcup\limits_{t\ge0}\Omega_t\times\{t\}
  \subset\R^{n+1}\times[0,\infty).
\end{equation}

Then we have
\begin{theorem}[Existence on bounded domains]
  \label{exist thm intro}
  Let $A\subset\R^{n+1}$ be a bounded open set and $u_0\colon A\to\R$
  a locally Lipschitz continuous function with $u_0(x)\to\infty$ for
  $x\to x_0\in\partial A$. \par
  Then there exists $(\Omega,u)$, where
  $\Omega\subset\R^{n+1}\times[0,\infty)$ is relatively open, such
  that $u$ solves graphical mean curvature flow
  \[\dot u=\sqrt{1+|Du|^2}\cdot \divergenz\left(\frac{Du}
    {\sqrt{1+|Du|^2}}\right)\quad\text{in }\Omega
  \setminus(\Omega_0\times\{0\}).\] The function $u$ is smooth for
  $t>0$ and continuous up to $t=0$, $\Omega_0=A$, $u(\cdot,0)=u_0$ in
  $A$ and $u(x,t)\to\infty$ as $(x,t)\to\partial\Omega$, where
  $\partial\Omega$ is the relative boundary of $\Omega$ in
  $\R^{n+1}\times[0,\infty)$.
\end{theorem}

Such smooth solutions yield weak solutions to mean curvature flow.  To
describe the relation, we use the measure theoretic boundary
$\partial^\mu\Omega_t$ as introduced in Section~\ref{lsf appendix}. We
have the following informal version of our main theorem concerning the
level set flow:
\begin{theorem}[Weak flow]
  \label{weak flow thm intro}
  Let $(A,u_0)$ and $(\Omega,u)$ be as in Theorem
  \ref{exist thm intro}. Assume that the level set evolution of
  $\partial\Omega_0$ does not fatten. Then it coincides with
  $(\partial^\mu\Omega_t)_{t\ge0}$. 
\end{theorem}
For the general version of our existence theorem see Theorem
\ref{exist thm}. Theorem \ref{mainlevelsettheo} is our main result
concerning the connection between the smooth graphical flow and the
weak flow (in the level set sense) of the projections.  In general, we
do not know whether the solutions $(\Omega,u)$ are level set
solutions. We notice, however, that such a statement would imply
uniqueness of $(\Omega,u)$ in Theorem \ref{exist thm}.

The previous theorems also provide a way to obtain a weak evolution of
a set $E\subset\R^{n+1}$ with $E=\partial A$ for some open set $A$:
Consider a function $u_0\colon A\to\R$ as described in Theorem
\ref{exist thm}, for example $u_0(x):=\frac1{\dist(x,\partial A)}
+|x|^2$, and apply our existence theorem. Then we define as the weak
evolution of $E$ the family $(\partial\Omega_t)_{t\ge0}$ with the
notation from above.

\subsection*{Illustrations}
We illustrate our main theorems by some figures. In the description,
we assume for the sake of simplicity that $\Omega_t=E_t$.

\begin{figure}[htb] 
 \includegraphics[height=5cm]{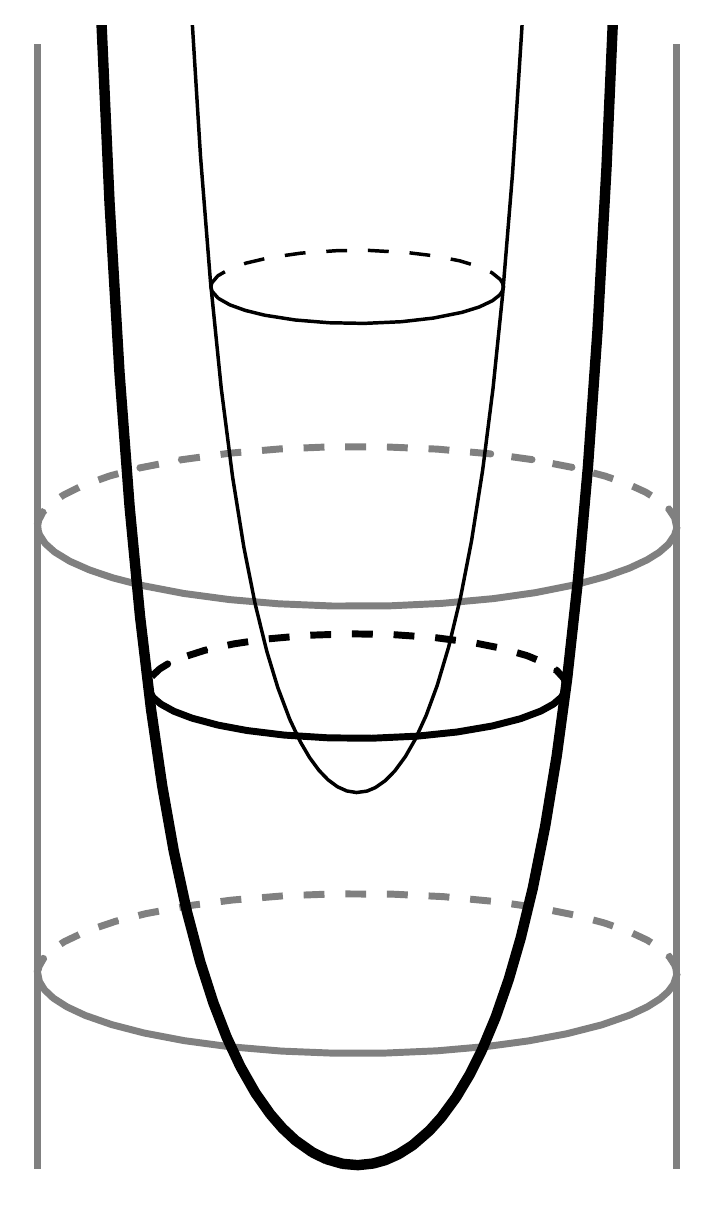}
 \caption{Graph over a ball}
 \label{bowl pic} 
\end{figure} 

In Figure \ref{bowl pic} we study the evolution of a graph over
$B_1(0)$ (drawn with thick lines), that is asymptotic to the cylinder
$\S^n\times\R$ (drawn with grey lines). The thinner lines indicate how
the graph looks at some later time. We remark that it continues to be
asymptotic to the evolving cylinder, which collapses in finite
time. As we prove in Theorem \ref{exist thm}, the evolving graph does
not become singular and it has to disappear to infinity at or before
the time the cylinder collapses. Theorem \ref{mainlevelsettheo}
implies that the evolving graph and the evolving cylinder disappear at
the same time. Notice that near the singular time, the lowest point
moves arbitrarily large distances in arbitrarily small time intervals.

\begin{figure}[htb] 
 \includegraphics[height=5cm]{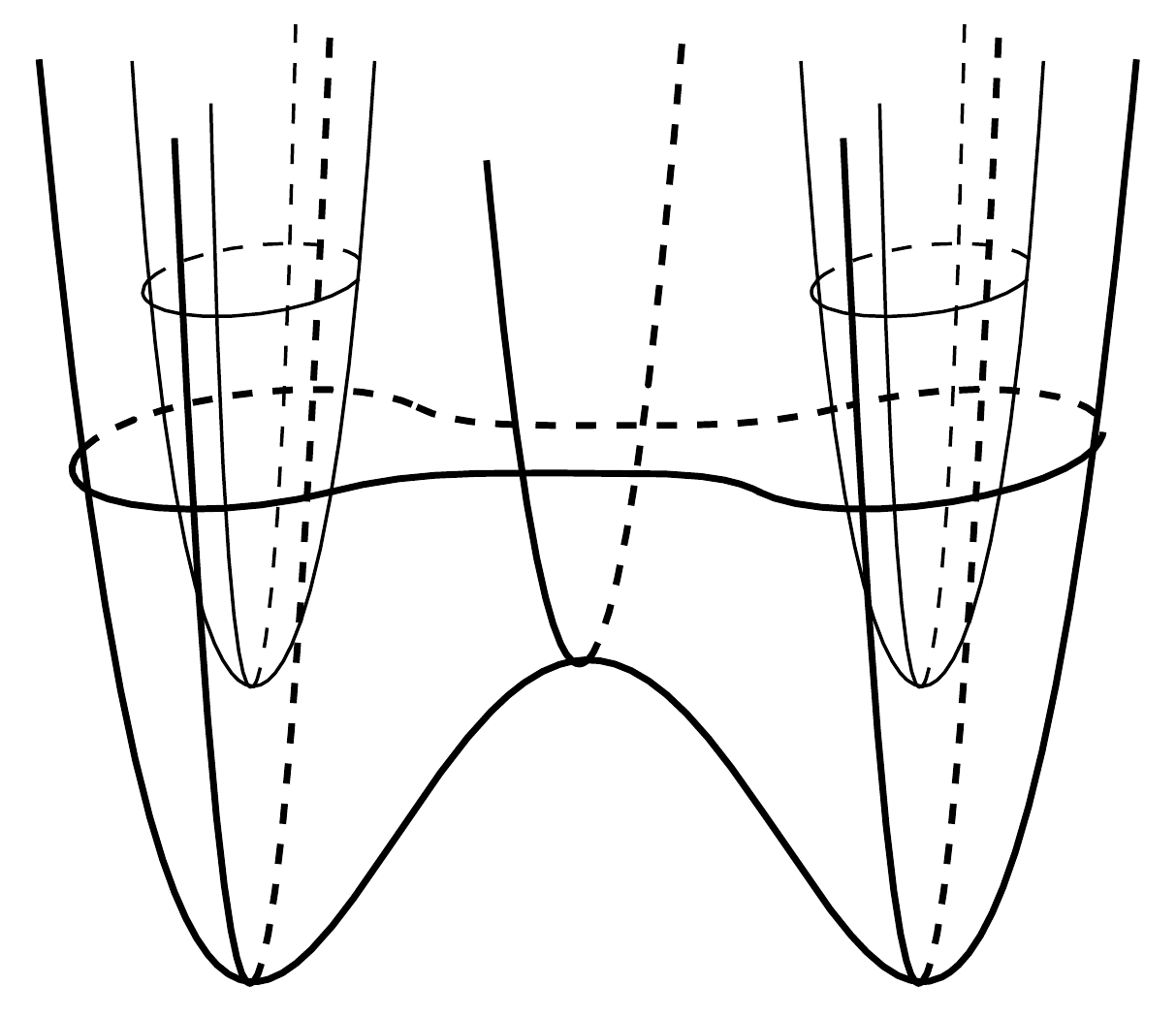}
 \caption{Graph over a set that develops a ``neck-pinch''}
 \label{neck pic} 
\end{figure} 

Figure \ref{neck pic} illustrates a graph over a set that develops a
``neck-pinch'' at $t=T$. This is projected to lower dimensions. For
$t\nearrow T$, the graph splits above the ``neck-pinch'' into two
disconnected components without becoming singular. The thinner lines
illustrate the graph for $t>T$. The rest of the evolution is similar
to the situation above.

\begin{figure}[htb] 
 \includegraphics[height=5cm]{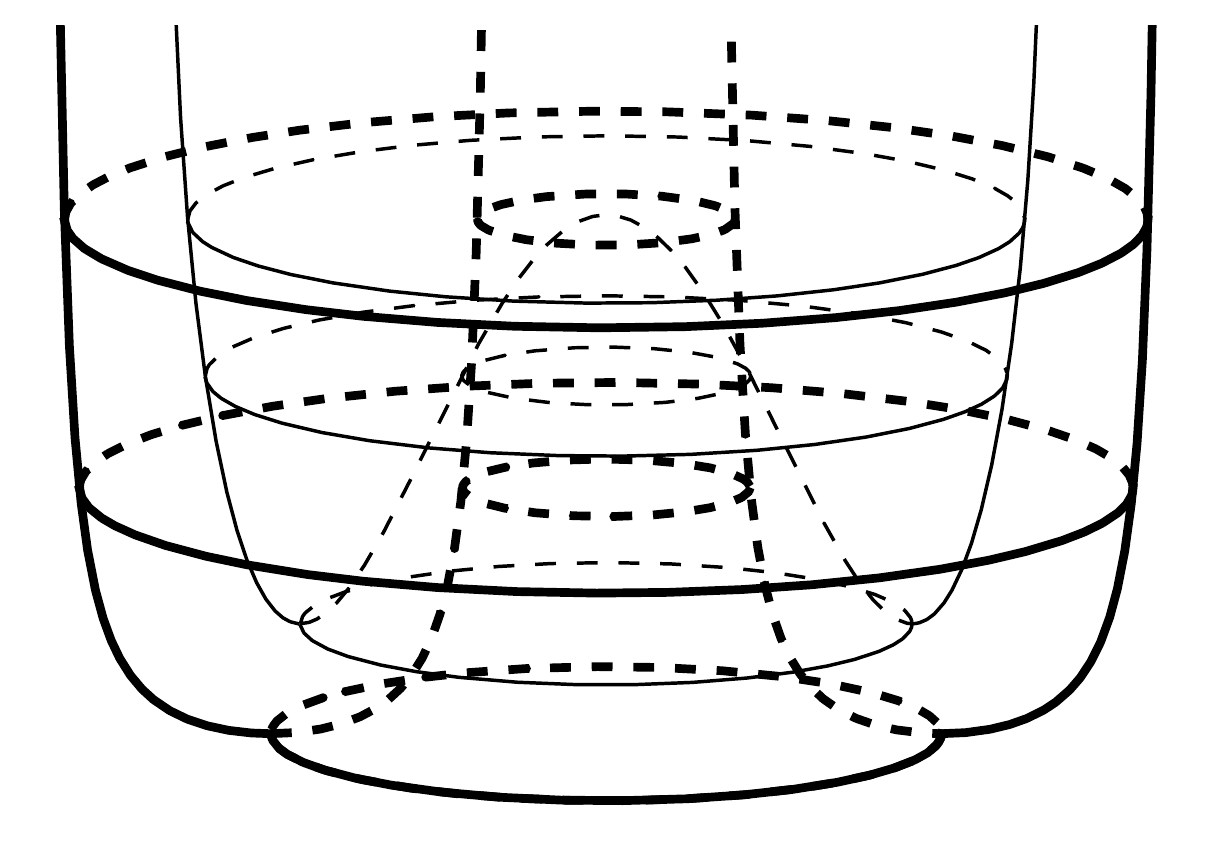}
 \caption{Graph defined initially over an annulus}
 \label{annulus pic} 
\end{figure} 

Next, we consider a rotationally symmetric graph over an annulus,
centered at the origin, see Figure \ref{annulus pic}. The inner
boundary of the annulus converges to a point as $t\nearrow T$. At
$t=T$ a ``cap at infinity'' is being added to the evolving graph. This
cap moves down very quickly. By comparison with compact solutions we
see that $u(0,t)$ is finite for any $t>T$. This is illustrated with
thin lines. Finally, once again the evolution becomes similar to the
evolution in Figure \ref{bowl pic}.

\begin{figure}[htb] 
 \includegraphics[height=5cm]{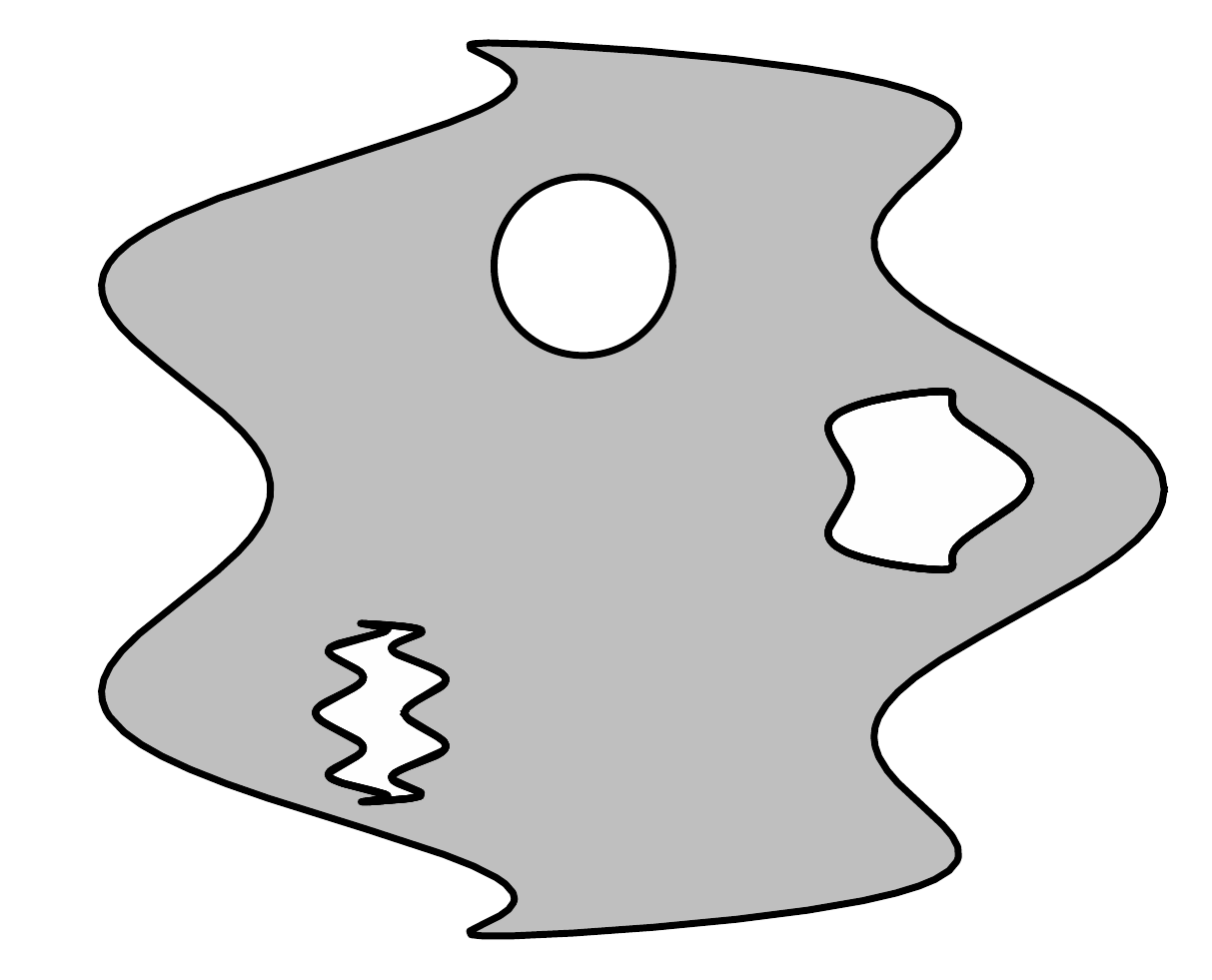}
 \caption{Domain with nontrivial topology}
 \label{cheese pic} 
\end{figure} 

Similarly, when a graph over a domain as in Figure \ref{cheese pic}
evolves, ``caps at infinity'' are being added at the times when the
small ``holes'' shrink to points.

\subsection*{Strategy of proof}
In order to prove existence of smooth solutions, we start by deriving
a priori estimates. The proof of these a priori estimates is based on
the observation that powers of the height function can be used to
localize derivative estimates in space. Then the result follows by
applying these estimates to approximate solutions and employing an
Arzel\`a-Ascoli-type theorem to pass to a limit. \par
The connection between singularity resolving and weak solutions is
obtained as follows: We observe that the cylinder
$(\partial\Omega_t\times\R)_t$ acts as an outer barrier for $\graph
u(\cdot,t)$. Furthermore, since $\graph u(\cdot,t)-R$ converges to the
cylinder as $R\to\infty$, we conclude that $\graph u(\cdot,t)$ does
not detach from the evolving cylinder near infinity.

\subsection*{Literature}
The existence of entire graphs evolving by mean curvature flow was
proved by K. Ecker and G. Huisken \cite{EckerHuiskenInvent} for
Lipschitz continuous initial data and by J. Clutterbuck
\cite{JulieAnisoMCF}, T. Colding and W. Minicozzi
\cite{ColdingMinicozzi} for continuous initial data. K. Ecker,
G. Huisken \cite{EckerHuiskenAnn} and N. Stavrou \cite{StavrouSelfSim}
have studied convergence to homothetically expanding solutions,
J. Clutterbuck, O. Schn\"urer, F. Schulze \cite{JCOSFSMCFStability}
and A. Hammerschmidt \cite{HammerschmidtPhD} have investigated
stability of entire solutions. \par
Many authors have worked on weak formulations for mean curvature flow,
e.\,g.{} K. Brakke \cite{brakke}, K. Ecker \cite{KEMCFBook},
L.\,{}C. Evans, J. Spruck
\cite{evanssprucklevelset1,evanssprucklevelset2,evanssprucklevelset3,evanssprucklevelset4},
Y. Chen, Y. Giga, S. Goto \cite{ChenGigaGotolevelset} and T. Ilmanen
\cite{TomMCF}. In what follows we will refer as {\it weak flow} to
level set solutions to mean curvature flow in the sense of Appendix
\ref{lsf appendix}, see also
\cite{ChenGigaGotolevelset,evanssprucklevelset1,JohnHeadPhD}.

Smooth solutions and one additional dimension have been used by
S. Altschuler, M. Grayson \cite{SteveThroughSing} for curves to
extend the evolution past singularities and by T. Ilmanen
\cite{TomEllReg} for the $\epsilon$-regularization of mean curvature
flow.\par
Several people have studied mean curvature flow after the first
singularity. We mention a few papers addressing this issue: J. Head
\cite{JohnHeadPhD} and J. Lauer \cite{LauerMCFSurgery} have shown that
an appropriate limit of mean curvature flows with surgery (see
G. Huisken and C. Sinestrari \cite{HuiskenSinestrari3} for the
definition of mean curvature flow with surgery) converges to a weak
solution. T. Colding and W. Minicozzi
\cite{ColdingMinicozziGenericMCF1} consider generic initial data that
develop only singularities that look spherical or cylindrical. In the
rotationally symmetric case, Y. Giga, Y. Seki and N. Umeda consider
mean curvature flow that changes topology at infinity
\cite{GigaOpenEnds,GigaQuenchingProfile}. \par
The height function has been used before in \cite{GT} to localize a
priori estimates for Monge-Amp\`ere equations.

\subsection*{Organization of the paper} The classical formulation
$\dot X=-H\nu$ of mean curvature flow does not allow for changes in
the topology of the evolving hypersurfaces. Hence in Section \ref{def
  sol sec} we introduce a notion of graphical mean curvature flow that
allows for changing domains of definition for the graph function and
hence also changes in the topology of the evolving submanifold.  \par
We fix our geometric notation in Section \ref{dg sec} and state
evolution equations of geometric quantities in Section \ref{evol eq
  sec}. \par
The key ingredients for proving smooth existence are the a priori
estimates in Section \ref{a priori sec} that use the height function
in order to localize the estimates. \par
In Section \ref{existence sec} we prove existence of smooth
solutions. That result follows from combining the H\"older estimates
of Section \ref{hoelder sec} and the compactness result that we prove
in Section \ref{compactness results sec} (a version of the Theorem of
Arzel\`a-Ascoli).  In Section \ref{levelset} we discuss the
relationship of our solution and the level set flow solution; we prove
Theorem \ref{mainlevelsettheo}. Finally, we include an appendix that
summarizes some of the results used in Section \ref{levelset}.

\subsection*{Open problems} 

We wish to mention a few open problems:
\begin{enumerate}
\item What is a good description of solutions disappearing at
  infinity?
\item If the projected solutions or a connected component of the
  complement become symmetric, e.\,g.{} spherical, does the graph pick
  up that symmetry?
\item What are optimal a priori estimates? 
\item Is the solution $(\Omega,u)$ unique? 
\item Does the level set solution of $\graph u_0$ fatten? Is this
  fattening related to that of the level set solution of $\partial A$?
\end{enumerate}

\subsection*{Acknowledgment}
We want to thank many colleagues for their interest in our work and
inspiring discussions: G. Bellettini, K. Ecker, G. Huisken,
T. Ilmanen, H. Koch, J. Metzger, F. Schulze, J. Spruck and B. White.
Some of these discussions were possible due to invitations to
Barcelona, Berlin, Oberwolfach and Potsdam.

\section{Definition of a solution}
\label{def sol sec}

\begin{definition}\label{mcf def}
  \neueZeile
  \begin{enumerate}[(i)]
  \item\label{mcf def i} \emph{Domain of definition:} Let
    $\Omega\subset\R^{n+1}\times[0,\infty)$ be a (relatively) open
    set. Set $\Omega_t:=\pi_{\R^{n+1}}\left(\Omega\cap
      \left(\R^{n+1}\times\{t\}\right)\right)$, where
    $\pi_{\R^{n+1}}\colon\R^{n+2}\to\R^{n+1}$ is the projection to the
    first $n+1$ components.  Notice here that the first $n+1$
    components of the domain $\Omega$ are spatial, while the
    last component can be understood as the time component.\par
    Observe that for each fixed $t$ the section
    $\Omega_t\subset\R^{n+1}$ is relatively open.
  \item\label{mcf def ii} \emph{The solution:} A function
    $u\colon\Omega\to\R$ is called a classical solution to graphical
    mean curvature flow in $\Omega$ with continuous initial value
    $u_0\colon\Omega_0\to\R$, if \[u\in C^{2;1}_{\text{loc}} (\Omega
    \setminus(\Omega_0\times\{0\})) \cap C^0_{\text{loc}} (\Omega)\]
    where we recall the definition of the spaces below and
    \begin{equation}
      \label{mcf eq}
      \tag{MCF}
      \begin{cases}
        \dot u=\sqrt{1+|Du|^2}\cdot\divergenz
        \left(\fracd{Du}{\sqrt{1+|Du|^2}}\right) &\text{in }\Omega
        \setminus(\Omega_0\times\{0\}),\\
        u(\cdot,0)=u_0&\text{in }\Omega_0.
      \end{cases}
    \end{equation}
  \item \emph{Maximality condition:} A function $u\colon\Omega\to\R$
    fulfills the maximality condition if $u\ge-c$ for some $c\in\R$
    and if $u|_{\Omega\cap\left(\R^{n+1}\times[0,T]\right)}$ is proper
    for every $T>0$. \par An initial value $u_0\colon\Omega_0\to\R$,
    $\Omega_0\subset\R^{n+1}$, is said to fulfill the maximality
    condition if $w\colon\Omega_0\times [0,\infty)\to\R$ defined by
    $w(x,t):=u_0(x)$ fulfills the maximality condition.
  \item \emph{Singularity resolving solution:} A function
    $u\colon\Omega\to\R$ is called a singularity resolving
    solution to mean curvature flow in dimension $n$ with initial
    value $u_0\colon\Omega_0\to\R$ if
    \begin{enumerate}[a)]
    \item $\Omega$ and $\Omega_0$ are as in \eqref{mcf def
        i},
    \item $u$ is a classical solution to graphical mean curvature flow
      with initial value $u_0$ as in \eqref{mcf def ii} and
    \item $u$ fulfills the maximality
      condition.
    \end{enumerate}
  \item We do not only call $u$ a singularity resolving solution but
    also the pair $(\Omega,u)$ and the family
    $(M_t)_{t\ge0}$ with $M_t=\graph u(\cdot,t)\subset\R^{n+2}$. 
  \end{enumerate}
\end{definition}

\begin{remark}
  \neueZeile
  \begin{enumerate}[(i)]
  \item Note that the domain of definition will depend on the
    solution.\par The dimensions seem to be artificially increased by
    one. This is due to the fact that we wish to study the evolution
    of $(\partial\Omega_t)_{t\ge0}$, which in the smooth case, see
    Remark \ref{lsf sing flow rem} \eqref{regular item}, is a family
    of $n$-dimensional hypersurfaces in $\R^{n+1}$ solving mean
    curvature flow.
  \item If $\Omega=\R^{n+1}$ then condition \eqref{mcf def ii} in
    Definition \ref{mcf def} coincides with the definition in
    \cite{EckerHuiskenInvent}.\par We avoid writing a solution as a
    family of embeddings $X\colon M\to\R^{n+2}$ as in general, the
    topology of $M$ is not fixed when $\Omega_t$ becomes
    singular. \par
    We expect similar results for other normal velocities, for
    example, if $u$ is a singularity resolving solution for the normal
    velocity $S_k$ in dimension $n$ then \[\dot u=\sqrt{1+|Du|^2}\cdot
    S_k[u]\quad\text{in }\Omega\setminus(\Omega_0\times\{0\}),\] where
    $S_k[u]$ denotes the $k$-th elementary symmetric function of the
    $n+1$ principal curvatures of graph $u(\cdot,t)\subset\R^{n+2}$
    and $\Omega$ is as in Definition \ref{mcf def} \eqref{mcf def i}.
  \item 
    \begin{enumerate}[a)]
    \item The maximality condition implies that $u$ tends to infinity
      if we approach a point in the relative boundary
      $\partial\Omega$. It also ensures that $u(x,t)$ tends to
      infinity as $|x|$ tends to infinity. \par
      Hence the maximality allows us to use the height function $u$
      for localizing our a priori estimates.
    \item Our maximality condition implies that each graph \[M_t
      =\graph u(\cdot,t)\subset\R^{n+2}\] is a complete
      submanifold. 
    \item If $u$ fulfills the maximality condition then
      $u_0(x):=u(x,0)$ also fulfills the maximality condition.
    \item The maximality condition prevents solutions from stopping or
      starting suddenly. Furthermore, in general restrictions of the
      domain $\Omega$ of a singularity resolving solution $(\Omega,u)$
      do not provide other singularity resolving solutions, i.\,e.{}
      for general open sets $B\subset\R^{n+1}\times[0,\infty)$, the
      pair $(\Omega\cap B,u|_B)$ does not fulfill the maximality
      condition. 
    \end{enumerate}
  \item It suffices to study classical solutions to mean curvature
    flow to obtain singularity resolving solutions. Nevertheless, this
    allows to obtain weak solutions starting with $\partial\Omega_0$
    by considering the projections of the evolving graphs.
  \end{enumerate}
\end{remark}

\section{Differential geometry of submanifolds}
\label{dg sec}

We use $X=X(x,\,t)=\left(X^\alpha\right)_{1\le\alpha\le n+2}$ to
denote the time-dependent embedding vector of a manifold $M^{n+1}$
into $\R^{n+2}$ and $\dt X=\dot X $ for its total time derivative.
Set $M_t:=X(M,\,t)\subset\R^{n+2}$. We will often identify an embedded
manifold with its image. We will assume that $X$ is smooth.  Assume
furthermore that $M^{n+1}$ is smooth, orientable, complete and
$\partial M^{n+1}=\emptyset$. We also use that notation if we have
that situation only locally, e.\,g.{} when the topology changes at
spatial infinity.\par
We choose $\nu=\nu(x)=\left(\nu^\alpha\right)_{1\le\alpha\le n+2}$ to
be the downward pointing unit normal vector to $M_t$ at $x$.  The
embedding $X(\cdot,\,t)$ induces at each point of $M_t$ a metric
$({g}_{ij})_{1\le i,\,j\le n+1}$ and a second fundamental form
$(h_{ij})_{1\le i,\,j\le n+1}$. Let $\left(g^{ij}\right)$ denote the
inverse of $(g_{ij})$. These tensors are symmetric and the principal
curvatures $(\lambda_i)_{1\le i\le n+1}$ are the eigenvalues of the
second fundamental form with respect to that metric. As usual,
eigenvalues are listed according to their multiplicity.\par
Latin indices range from $1$ to $n+1$ and refer to geometric
quantities on the surface, Greek indices range from $1$ to $n+2$ and
refer to components in the ambient space $\R^{n+2}$.  In $\R^{n+2}$,
we will always choose Euclidean coordinates with fixed $e_{n+2}$-axis.
We use the Einstein summation convention for repeated upper and lower
indices. Latin indices are raised and lowered with respect to the
induced metric or its inverse $\left(g^{ij}\right)$, while for Greek
indices we use the flat metric $(\ol
g_{\alpha\beta})_{1\le\alpha,\beta\le n+2}
=(\delta_{\alpha\beta})_{1\le\alpha,\beta\le n+2}$ of $\R^{n+2}$. \par
Denoting by $\langle\cdot,\,\cdot\rangle$ the Euclidean scalar product
in $\R^{n+1}$, we have $$g_{ij}=\left\langle
  X_{,\,i},\,X_{,\,j}\right\rangle
=X^\alpha_{,\,i}\delta_{\alpha\beta}X^\beta_{,\,j},$$ where we use
indices preceded by commas to denote partial derivatives. We write
indices preceded by semi-colons, e.\,g.\ $h_{ij;\,k}$ or $v_{;k}$, to
indicate covariant differentiation with respect to the induced metric.
Later, we will also drop the semi-colons and commas, if the meaning is
clear from the context. We set $X^\alpha_{;i}\equiv X^\alpha_{,i}$ and
\begin{equation}
  \label{eq:Xij}
  X^\alpha_{;\,ij}=X^\alpha_{,\,ij}-\Gamma^k_{ij}X^\alpha_{,\,k},
\end{equation}
where 
$$\Gamma^k_{ij}=\tfrac12g^{kl}(g_{il,\,j}+g_{jl,\,i}-g_{ij,\,l})$$
are the Christoffel symbols of the metric $(g_{ij})$.  So
$X^\alpha_{;ij}$ becomes a tensor. \par
The Gau{\ss} formula relates covariant derivatives of the position
vector to the second fundamental form and the normal vector
\begin{equation}\label{Gauss formula}
X^\alpha_{;\,ij}=-h_{ij}\nu^\alpha.
\end{equation}
\par 
The Weingarten equation allows to compute derivatives of the normal
vector
\begin{equation}\label{Weingarten equation}
\nu^\alpha_{;\,i}=h^k_iX^\alpha_{;\,k}.
\end{equation}
\par
We can use the Gau{\ss} formula \eqref{Gauss formula} or the
Weingarten equation \eqref{Weingarten equation} to compute the second
fundamental form. \par
Symmetric functions of the principal curvatures are well-defined, we
will use the mean curvature $H=\lambda_1+\ldots+\lambda_{n+1}$ and the
square of the norm of the second fundamental form
$|A|^2=\lambda_1^2+\ldots+\lambda_{n+1}^2$. \par
Our sign conventions imply that $H>0$ for the graph of a strictly
convex function.
\par
The space $C^{k,\alpha;k/2,\alpha/2}$ denotes the space of functions
for which up to $k$-th derivatives are continuous, where time
derivatives count twice, these derivatives are H\"older continuous
with exponent $\alpha$ in space and $\alpha/2$ in time and the
corresponding H\"older norm is finite. The space $C^k_{loc}(\Omega)$
consists of the functions $u\colon\Omega\to\R$ which are in $C^k(K)$
for every $K\Subset\Omega$. We use similar definitions for other
(H\"older) spaces.
\par
Finally, we use $c$ to denote universal, estimated constants. 

\section{Evolution equations for mean curvature flow}
\label{evol eq sec}

\begin{definition}
  If $M$ is given as an embedding and a graph, we use
  $\eta=(0,\ldots,0,1)$ to denote the vector $e_{n+2}$. The
  definitions of $\nu$, $H$ and $|A|^2$ are as introduced in the
  previous section. We denote the induced connection by $\nabla$
  and the associated Laplace-Beltrami operator by $\Delta$.
  
  We define $v=\left(-\eta_\alpha\nu^\alpha\right)^{-1}$ and
  $u=\eta_\alpha X^\alpha$. The function $u$ can be regarded as a
  function defined on a subset of $\R^{n+1}\times[0,\infty)$ or as a
  function defined on the evolving manifold $M$. It should be clear
  from the context which definition of $u$ is being used.
\end{definition}

\begin{theorem}\label{evol eq thm}
  Let $X$ be a solution to mean curvature flow. Then we have the
  following evolution equations. 
  \begin{align*}
    \heat u=&\,0,\umbruch\\
    \heat v=&\,-v|A|^2-\tfrac2v|\nabla v|^2,\umbruch\\
    \heat |A|^2=&\,-2|\nabla A|^2+2|A|^4,\umbruch\\
    \heat \left|\nabla^mA\right|^2
    \le&\,-2\left|\nabla^{m+1}A\right|^2\\
    &\,\quad +c(m,n)\cdot\sum\limits_{i+j+k=m}
    \left|\nabla^mA\right|\cdot \left|\nabla^iA\right|\cdot
    \left|\nabla^jA\right|\cdot \left|\nabla^kA\right|,\umbruch\\
    \heat \G\le&\,-2k\G^2 -2\phi v^{-3}\langle\nabla
    v,\nabla \G\rangle,
  \end{align*}
  where $\G=\phi|A|^2 \equiv\frac{v^2}{1-kv^2}|A|^2$ and $k>0$ is
  chosen so that $kv^2\le\frac12$ in the domain considered.
\end{theorem}
We remark that whenever we use evolution equations from this theorem,
we consider $u$ as a function defined on the evolving manifold.
\begin{proof}
  See \cite{KEMCFBook,EckerHuiskenInvent}.
\end{proof}

\section{A priori estimates}
\label{a priori sec}

The following assumption shall guarantee that we can prove local a
priori estimates for the part of $\graph u$ where $u<0$.  Notice that,
via considering the evolution given by $u-a$ (where $a$ is a constant
abbreviating the Spanish word ``altura''), this is equivalent to
obtain bounds in the set where $u<a$.

In this section we will consider the set $\hat\Omega=
\{u<0\}$. More precisely, we will work under the following
assumption:
\begin{assumption}\label{est ass}
  Let $\hat\Omega\subset\R^{n+1}\times[0,\infty)$ be an open set. Let
  $u\colon\hat\Omega\to\R$ be a smooth graphical solution to \[\dot
  u=\sqrt{1+|Du|^2} \cdot\divergenz
  \left(\frac{Du}{\sqrt{1+|Du|^2}}\right)\quad\text{in}
  \quad\hat\Omega\cap\left(\R^{n+1}\times(0,\infty)\right).\] Suppose
  that $u(x,t)\to0$ as $(x,t)\to(x_0,t_0)\in\partial\hat\Omega$.
  Assume that all derivatives of $u$ are uniformly bounded and can be
  extended continuously across the boundary for all domains
  $\hat\Omega\cap\left(\R^{n+1}\times[0,T]\right)$ and that these sets
  are bounded for any $T>0$.
\end{assumption}

\begin{remark}
  \neueZeile
  \begin{enumerate}[(i)]
  \item Assumption \ref{est ass} is fulfilled for smooth entire
    solutions $u$ to graphical mean curvature flow that fulfill $u\ge
    L\ge 1$ outside a compact set when we restrict $u$ to
    $\hat\Omega=\left\{(x,t)\in\R^{n+1} \times[0,\infty)\colon
      u(x,t)<0\right\}$.
  \item The approximate solutions $u^L_{\epsilon,R}$ in Lemma \ref{ex
      approx sol lem} fulfill Assumption \ref{est ass} for $L>0$.
  \item The following a priori estimates extend to the situation when
    \[\hat\Omega =\{(x,t)\colon u(x,t)<a\}\] for any $a\in\R$ instead of
    $0$. We only have to replace $u$ by $(u-a)$ below, e.\,g.{} in
    Theorem \ref{C1 est thm}.
  \item The boundedness assumption of the sets follows from the
    properness of the function $u$.
  \end{enumerate}
\end{remark}

\begin{theorem}[$C^1$-estimates]
  \label{C1 est thm}
  Let $u$ be as in Assumption \ref{est ass}. Then
  \[vu^2\le\max\limits_{\genfrac{}{}{0pt}{}{t=0}{\{u<0\}}}vu^2\]
  at points where $u<0$.
\end{theorem}
Here and in what follows, it is often possible to increase the
exponent of $u$.
\begin{proof}
  According to Theorem \ref{evol eq thm}, $w:=vu^2$ fulfills
  \begin{align*}
    \dot w=&\,\dot vu^2+2vu\dot u,\umbruch\\
    w_i=&\,v_iu^2+2vuu_i,\umbruch\\
    w_{ij}=&\,v_{ij}u^2+2vuu_{ij} +2vu_iu_j
    +2u(v_iu_j+v_ju_i),\umbruch\\ 
    \heat w=&\,u^2\heat v
    -2v|\nabla u|^2-4u\langle\nabla v,\nabla u\rangle\\
    =&\,u^2\left(-v|A|^2-\tfrac2v|\nabla v|^2\right)
    -2v|\nabla u|^2 -4\left\langle\tfrac u{\sqrt v}\nabla v,\sqrt
      v\nabla u\right\rangle\umbruch\\
    \le&\,-u^2v|A|^2\le0.
  \end{align*}
  The estimate follows from the maximum principle applied to $w$ in
  the domain where $u<0$. 
\end{proof}

\begin{remark}
  If the reader prefers to consider a positive cut-off function
  $(-u)$, we recommend to rewrite Theorem \ref{C1 est thm} as an
  estimate for $v\cdot(-u)^2$.
\end{remark}

\begin{corollary}\label{C1 cor}
  Let $u$ be as in Assumption \ref{est ass}. Then
  \[v\le\max\limits_{\genfrac{}{}{0pt}{}{t=0}{\{u<0\}}}vu^2\] 
  at points where $u\le-1$. 
\end{corollary}

\begin{remark}
  Similar corollaries also hold for higher derivatives. We do not
  write them down explicitly. 
\end{remark}

\begin{remark}\label{nabla u v est rem}
  For later use, we estimate derivatives of $u$ and $v$,
  \begin{align*}
    |\nabla u|^2=&\,\eta_\alpha X^\alpha_ig^{ij}X^\beta_j\eta_\beta
    =\eta_\alpha\left(\delta^{\alpha\beta}
      -\nu^\alpha\nu^\beta\right)\eta_\beta =1-v^{-2}\le 1
    \intertext{and, according to \eqref{Weingarten equation},} |\nabla
    v|^2=&\,\left(\left(-\eta_\alpha\nu^\alpha\right)^{-1}\right)_i
    g^{ij}\left(\left(-\eta_\beta\nu^\beta\right)^{-1}\right)_j
    =v^4\eta_\alpha
    X^\alpha_kh^k_ig^{ij} h^l_jX^\beta_l\eta_\beta \le v^4|A|^2\\
    \le&\,v^2\phi|A|^2 =v^2\G.  \intertext{So we get} |\langle\nabla
    u,\nabla v\rangle| \le&\,|\nabla u|\cdot|\nabla v| \le v^2|A|\le
    v\sqrt \G.
  \end{align*}
\end{remark}

\begin{theorem}[$C^2$-estimates]
  \label{C2 est thm}
  Let $u$ be as in Assumption \ref{est ass}. 
  \begin{enumerate}[(i)]
  \item \label{only lip C2 est case}
    Then there exist $\lambda>0$, $c>0$ and $k>0$ (the constant in
    $\phi$ and implicitly in $\G$), depending on the $C^1$-estimates,
    such that
    \[tu^4\G+\lambda
    u^2v^2\le\sup\limits_{\genfrac{}{}{0pt}{}{t=0}{\{u<0\}}} \lambda
    u^2v^2 +ct\] at points where $u<0$ and $0<t\le 1$. 
  \item Moreover, if $u$ is in $C^2$ initially, we get $C^2$-estimates
    up to $t=0$: Then there exists $c>0$, depending only on the
    $C^1$-estimates, such
    that \[u^4\G\le\sup\limits_{\genfrac{}{}{0pt}{}{t=0}{\{u<0\}}}u^4\G
    +ct\] at points where $u<0$.
  \end{enumerate}
\end{theorem}
\begin{proof}
  In order to prove both parts simultaneously, we underline terms and
  factors that can be dropped everywhere. We get the first part if we
  consider the underlined terms and the second part if we drop those
  and set $\lambda=0$.\par 
  We set \[w:=\ul tu^4\G+\lambda u^2v^2\] and obtain
  \begin{align*}
    \dot w=&\,\ul{u^4 \G} +4\ul tu^3\G\dot u+\ul tu^4\dot \G +2\lambda
    v^2u\dot u +2\lambda u^2v\dot v,\umbruch\\
    w_i=&\,4\ul tu^3\G u_i +\ul tu^4\G_i +2\lambda v^2uu_i +2\lambda
    u^2vv_i,\umbruch\\
    w_{ij}=&\,4\ul tu^3\G u_{ij} +\ul tu^4\G_{ij} +2\lambda v^2uu_{ij}
    +2\lambda u^2vv_{ij} +12\ul tu^2\G u_iu_j\\ &\,+4\ul
    tu^3(\G_iu_j+\G_ju_i) +2\lambda v^2u_iu_j +2\lambda u^2v_iv_j
    +4\lambda vu(u_iv_j+u_jv_i),\umbruch\\
    \ul tu^3\nabla \G=&\,\frac1u\nabla w -4\ul tu^2\G\nabla u
    -2\lambda
    v^2\nabla u -2\lambda uv\nabla v,\umbruch\\
    \heat w\le&\,\ul{u^4\G} +\ul tu^4\left(-2k\G^2-2\phi
      v^{-3}\langle\nabla v,\nabla \G\rangle\right)
    +2\lambda u^2v\left(-v|A|^2 -\tfrac2v|\nabla
      v|^2\right)\\ &\,-12\ul tu^2\G|\nabla u|^2 -8\ul
    tu^3\langle\nabla \G,\nabla u\rangle -2\lambda v^2|\nabla u|^2
    -2\lambda u^2|\nabla v|^2\\ &\,-8\lambda uv\langle\nabla u,\nabla
    v\rangle.
  \end{align*}
  In the following, we will use the notation $\langle\nabla
  w,b\rangle$ for generic gradient terms for the test function
  $w$. The constants $c$ are allowed to depend on $\sup\{|u|\colon
  u<0\}$ (which does not exceed its initial value) and the
  $C^1$-estimates which are uniform as we may consider $v\cdot(u-1)^2$
  in Theorem \ref{C1 est thm}. In case \eqref{only lip C2 est case},
  it may also depend on an upper bound for $t$, but we assume that
  $0<t\le 1$. That is, we suppress dependence on already estimated
  quantities. \par We estimate the terms involving $\nabla \G$
  separately. Let $\epsilon>0$ be a constant. We fix its value
  blow. Using Remark \ref{nabla u v est rem} for estimating terms, we
  get
  \begin{align*}
    -2\phi\ul tu^4v^{-3}\langle\nabla v,\nabla \G\rangle
    =&\,-2\frac{\phi u}{v^3}\left\langle\nabla v,\frac1u\nabla w-4\ul
      tu^2\G\nabla u-2\lambda v^2\nabla u-2\lambda uv\nabla
      v\right\rangle\umbruch\\
    \le&\,\langle\nabla w,b\rangle +8\ul t\frac{\phi u^3}v\G|A|
    +4\lambda\phi |u|v|A| +4\frac{\lambda\phi u^2}{v^2} |\nabla
    v|^2\umbruch\\
    \le&\,\langle\nabla w,b\rangle +\epsilon\ul tu^4\G^2
    +\epsilon\lambda u^2v^2|A|^2 +\lambda u^2|\nabla v|^2\cdot
    4\frac{\phi}{v^2} +c(\epsilon,\lambda),\umbruch\\
    -8\ul tu^3\langle\nabla \G,\nabla u\rangle
    =&\,-8\left\langle\nabla u,\frac1u\nabla w -4\ul tu^2\G\nabla u
      -2\lambda v^2\nabla u
      -2\lambda uv\nabla v\right\rangle\umbruch\\
    \le&\,\langle\nabla w,b\rangle +32\ul tu^2\G+16\lambda
    v^2+16\lambda |u|v^3|A|\umbruch\\
    \le&\,\langle\nabla w,b\rangle +\epsilon\ul tu^4\G^2
    +\epsilon\lambda u^2v^2|A|^2 +c(\epsilon,\lambda).
  \end{align*}
  We obtain 
  \begin{align*}
    \heat w\le&\,\ul{u^4 \G} +\ul tu^4\G^2(-2k+2\epsilon)
    +\langle\nabla w,b\rangle\\ &\,+\lambda u^2v^2|A|^2(-2+3\epsilon)
    +\lambda u^2|\nabla v|^2\left(4\frac\phi{v^2}-6\right)
    +c(\epsilon,\lambda). 
  \end{align*}
  Let us assume that $k>0$ is chosen so small that $kv^2\le\frac13$ in
  $\{u<0\}$. This implies $\phi\le2v^2$.  We may assume that
  $\lambda\ge 2u^2$ in $\{u<0\}$ and get $u^4\G\le\frac12\lambda
  u^2\phi|A|^2 \le\lambda u^2v^2|A|^2$. We get \[4\frac\phi{v^2}-6
  =\frac4{1-kv^2} -6\le 0.\] Finally, fixing $\epsilon>0$ sufficiently
  small, we obtain \[\heat w\le\langle\nabla w,b\rangle +c.\] Now,
  both claims follow from the maximum principle.
\end{proof}

\begin{theorem}[$C^{m+2}$-estimates]
  \label{Ck est thm}
  Let $u$ be as in Assumption \ref{est
    ass}.
  \begin{enumerate}[(i)]
  \item There exists $\lambda>0$, depending on the
    $C^{m+1}$-estimates, such that     
    \[tu^2\left|\nabla^mA\right|^2
    +\lambda\left|\nabla^{m-1}A\right|^2 \le c\cdot\lambda\cdot
    t+\sup\limits_{\genfrac{}{}{0pt}{}{t=0}{\{u<0\}}}
    \lambda\left|\nabla^{m-1}A\right|^2 \] at points where $u<0$ and
    $0<t\le 1$.
  \item As in Theorem \ref{C2 est thm}, initial smoothness
    is preserved.
    \end{enumerate}
\end{theorem}

\begin{remark}
  \neueZeile
  \begin{enumerate}[(i)]
  \item This implies a priori estimates for arbitrary derivatives and
    any $t>0$: It is known that estimates for $u$, $v$, $|A|$ and
    $\left|\nabla^m A\right|$, $1\le m\le M$, imply (spatial)
    $C^{M+2}$-estimates for the function that represents the evolving
    hypersurface as a graph. Using the equation, we can bound time
    derivatives.
  \item For estimates at time $t_0>1$, we can use the previous
    theorems with $t=0$ replaced by $t=t_0-1/2$.
  \item To control the $m$-th (spatial) derivative at time $t_0>0$, we
    can apply the result iteratively and control the $k$-th
    derivatives, $1\le k\le m$, at time $\frac{kt_0}m$.
  \item Theorem \ref{Ck est thm} implies smoothness for $t>0$. We do
    not expect, however, that the decay rates obtained for $\Amod{m}$
    are optimal near $t=0$.
  \end{enumerate}
\end{remark}

\begin{proof}[Proof of Theorem \ref{Ck est thm}]
  Once again, we underline terms and factors that can be dropped to
  obtain uniform estimates up to $t=0$. We define \[w:=\ul
  tu^2\left|\nabla^mA\right|^2 +\lambda\left|\nabla^{m-1}A\right|^2\]
  for a constant $\lambda>0$ to be fixed. We will assume that
  $\left|\nabla^kA\right|^2$ is already controlled for any $0\le k\le
  m-1$. Suppose that $0\le t\le 1$. The constant $c$ is allowed to
  depend on quantities that we have already controlled. Thus the
  evolution equation for $\Amod{m}$ in Theorem \ref{evol eq thm}
  becomes for $m\ge1$
  \begin{align*}
    \heat\Amod{m}\le&\,-2\Amod{m+1} +c\Amod{m} +c,\umbruch\\
    \heat\Amod{m-1}\le&\,-2\Amod{m} +c.
  \end{align*}
  We get
  \begin{align*}
    \dot w=&\,\underline{u^2\Amod{m}} +2\ul tu\dot u\Amod{m} +\ul
    tu^2\dt\Amod{m} +\lambda\dt\Amod{m-1},\umbruch\\
    w_i=&\,2\ul tuu_i\Amod{m} +\ul tu^2\left(\Amod{m}\right)_i
    +\lambda\left(\Amod{m-1}\right)_i,\umbruch\\
    w_{ij}=&\,2\ul tuu_{ij}\Amod{m} +\ul
    tu^2\left(\Amod{m}\right)_{ij}
    +\lambda\left(\Amod{m-1}\right)_{ij}\\
    &\, +2\ul tu_iu_j\Amod{m} +2\ul tu\left(u_i\left(\Amod{m}\right)_j
      +u_j\left(\Amod{m}\right)_i\right),\umbruch\\
    \heat w\le&\,\ul{u^2\Amod{m}} +\ul
    tu^2\left(-2\Amod{m+1}+c\Amod{m}+c\right)\\
    &\,+\lambda\left(-2\Amod{m}+c\right) -2\ul t|\nabla
    u|^2\Amod{m} -4\ul tu\left\langle\nabla u,
      \nabla\Amod{m}\right\rangle.
  \end{align*}
  We observe that 
  \[-4\ul tu\left\langle\nabla u, \nabla\Amod{m}\right\rangle \le \ul
  t\cdot |u|\cdot c\cdot\left|\nabla^{m+1}A\right|
  \cdot\left|\nabla^mA\right| \le\ul tu^2\Amod{m+1} +c\Amod{m}.\]
  Therefore we get
  \[\heat w\le\left(c-2\lambda\right)\Amod{m} +c(\lambda)\]
  and the result follows from the maximum principle. 
\end{proof}

\section{H\"older estimates in time}
\label{hoelder sec}

We will use the following H\"older estimates to prove maximality of a
limit of solutions. 
\begin{lemma}\label{hoelder t lem}
  Let $u\colon\R^{n+1}\times[0,\infty)\to\R$ be a graphical solution
  to mean curvature flow and $M\ge1$ such that \[|Du(x,t)|\le M\quad
  \text{for all $(x,t)$ where }\quad u(x,t)\le0.\] Fix any
  $x_0\in\R^{n+1}$ and $t_1,t_2\ge0$. If $u(x_0,t_1)\le-1$ or
  $u(x_0,t_2)\le-1$, then $|t_1-t_2|\ge\frac1{8(n+1)M^2}$ or
  \[\frac{|u(x_0,t_1)-u(x_0,t_2)|}{\sqrt{|t_1-t_2|}}
  \le\sqrt{2(n+1)}(M+1).\]   
\end{lemma}
The previous lemma implies that $u$ is locally uniformly H\"older
continuous in time. Although Lemma \ref{hoelder sec} follows from the
bounds for $H$ provided by \cite[Theorem 3.1]{EckerHuiskenInvent}, we
include below an independent and more elementary proof which employs
spheres as barriers.
\begin{proof}
  We may assume that $t_1\le t_2$. 
  \begin{enumerate}[(i)]
  \item\label{max lem i} Assume first that $u(x_0,t_1)\le-1$. As
    $|Du(x,t)|\le M$ for $u(x,t)\le0$, we deduce for any
    $0<r\le\frac1M$ \[u(x_0,t_1)-Mr\le u(y,t_1)\le
    u(x_0,t_1)+Mr\quad\text{for all }y\in B^{n+1}_r(x_0).\] Hence the
    sphere $\partial B_r^{n+2}(x_0,u(x_0,t_1)+(M+1)r)$ lies above
    $\graph u(\cdot,t_1)$ and $\partial
    B_r^{n+2}(x_0,u(x_0,t_1)-(M+1)r)$ lies below $\graph
    u(\cdot,t_1)$. When the spheres evolve by mean curvature flow,
    their radii are given by \[r(t)=\sqrt{r^2-2(n+1)(t-t_1)}\] for
    $t_1\le t<t_1+\frac{r^2}{2(n+1)}$. Both spheres are compact
    solutions to mean curvature flow. Hence they are barriers for
    $\graph u(\cdot,t)$. In particular, we get \[u(x_0,t_1)-(M+1)r \le
    u\left(x_0,t_1+\frac{r^2}{2(n+1)}\right) \le u(x_0,t_1)+(M+1)r.\]
    Set $r:=\sqrt{2(n+1)(t_2-t_1)}$. We may assume
    $|t_1-t_2|\le\frac1{2(n+1)M^2}$. Hence $r\le\frac1M$ and the
    considerations above apply. We obtain
    \begin{align*}
    u(x_0,t_1)-(M+1)\sqrt{2(n+1)(t_2-t_1)} \le&\, u(x_0,t_2)\\ \le&\,
    u(x_0,t_1)+(M+1)\sqrt{2(n+1)(t_2-t_1)}.      
    \end{align*}
    Rearranging implies the H\"older continuity claimed above.
  \item Assume now that $u(x_0,t_2)\le-1$ and $u(x_0,t_1)>-1$. We
    argue by contradiction: Suppose that $t_2\ge t_1\ge
    t_2-\frac1{8(n+1)M^2}$ and
    \begin{equation}
      \label{hoelder violated eq}
      \frac{u(x_0,t_1)-u(x_0,t_2)}
      {\sqrt{t_2-t_1}}\ge\sqrt{2(n+1)}(M+1).      
    \end{equation}
    Set $r:=\sqrt{2(n+1)(t_2-t_1)}$. We claim that
    \begin{equation}
      \label{sphere below eq}
      \min\{u(x_0,t_1),0\}-Mr\ge u(x_0,t_2)+r.
    \end{equation}
    If $u(x_0,t_1)<0$, \eqref{sphere below eq} follows by rearranging
    \eqref{hoelder violated eq}. Otherwise, we have that
    \begin{align*}
      u(x_0,t_2)+(M+1)r\le&\,-1+(M+1)\sqrt{2(n+1)(t_2-t_1)}\\
      \le&\,-1+(M+1)\sqrt{\frac{2(n+1)}{8(n+1)M^2}}
      \le-1+\frac{M+1}{2M}\le0 
    \end{align*}
    as $M\ge1$. This proves claim \eqref{sphere below eq}. \par
    Now, using \eqref{sphere below eq}, we can proceed similarly as in
    \eqref{max lem i}: For some small $\epsilon>0$, the sphere
    $\partial B^{n+2}_r(x_0,u(x_0,t_2)+\epsilon)$ lies below $\graph
    u(\cdot,t_1)$ (for the positivity of $\epsilon$ consider in
    \eqref{sphere below eq} the terms $-Mr$ near the center and $+r$
    near the boundary). Under mean curvature flow, the sphere shrinks
    to a point as $t\nearrow t_2$ and stays below $\graph
    u(\cdot,t)$. We obtain $u(x_0,t_2)+\epsilon\le u(x_0,t_2)$, which
    is a contradiction. \qedhere
  \end{enumerate}
\end{proof}

\section{Compactness results}
\label{compactness results sec} 

\begin{lemma}
  \label{maximality lem}
  Let $\Omega\subset B\subset\R^{n+2}$ and consider a function
  $u\colon\Omega\to\R$. Assume that for each $a\in\R$ there exists
  $r(a)>0$ such that for each $x\in\Omega$ with $u(x)\le a$ we have
  $B_{r(a)}(x)\cap B\subset\Omega$. Then $\Omega$ is relatively open
  in $B$ and $u(x_k)\to\infty$ if $x_k\to x\in\partial\Omega$, where
  $\partial\Omega$ is the relative boundary of $\Omega$ in $B$.
\end{lemma}
\begin{proof}
  It is clear that $\Omega\subset B$ is relatively open. If $u$ were
  not tending to infinity near the boundary, we find $x_n\in\Omega$
  such that $x_n\to x\in\partial\Omega$ as $n\to\infty$ and $u(x_n)\le
  a$ for some $a\in\R$. Since $B_{r(a)}(x_n)\cap B\subset\Omega$, the
  triangle inequality implies $x\in B_{r(a)}(x_n)$ for $n$
  sufficiently large. This contradicts $x\in\partial\Omega$.
\end{proof}

\begin{remark}
  A continuous maximal graph is a closed set and -- if sufficiently
  smooth -- a complete manifold. 
\end{remark}

\begin{lemma}[Variation on the Theorem of Arzel\`a-Ascoli]
  \label{AA lem}
  Let $B\subset\R^{n+2}$ and $0<\alpha\le 1$. Let $u_i\colon
  B\to\R\cup\{\infty\}$ for $i\in\N$. Suppose that there exist
  strictly decreasing functions $r,\,-c\colon\R\to\R_+$ such that for
  each $x\in B$ and $i\ge i_0(a)$ with $u_i(x)\le a<\infty$ we
  have \[\frac{|u_i(x)-u_i(y)|}{|x-y|^\alpha}\le c(a)\quad\text{for
    all}\quad y\in\ol{B_{r(a)}(x)}\cap B.\] Then there exists a
  function $u\colon B\to\R\cup\{\infty\}$ such that a subsequence
  $(u_{i_k})_{k\in\N}$ converges to $u$ locally uniformly in
  $\Omega:=\{x\in B\colon u(x)<\infty\}$ and $u_{i_k}(x)\to\infty$ for
  $x\in B\setminus\Omega$.  Moreover, for each $x\in\Omega$ with
  $u(x)\le a$ we have $B_{r(a+1)}(x)\cap B\subset\Omega$
  and \[\frac{|u(x)-u(y)|}{|x-y|^\alpha}\le c(a+1)\quad\text{for
    all}\quad y\in\ol{B_{r(a+1)}(x)}\cap B.\]
\end{lemma}
\begin{proof}
  We adapt the proof of the Theorem of Arzel\`a-Ascoli to our
  situation. Let $D:=\{x_l\colon l\in\N\}$ be dense in $B$.\par
  If $\liminf\limits_{i\to\infty}u_i(x_0)<\infty$, we choose a
  subsequence $(u_{i_k})_{k\in\N}$, such that
  $\lim\limits_{k\to\infty}u_{i_k}(x_0)
  =\liminf\limits_{i\to\infty}u_i(x_0)$. If
  $\liminf\limits_{i\to\infty}u_i(x_0)=\infty$, we do not need to pass
  to a subsequence. \par Proceed similarly with $x_1,x_2,\ldots$
  instead of $x_0$. We denote the diagonal sequence of this sequence
  of subsequences by $(\tilde u_i)_{i\in\N}$. Define
  $u(x_k):=\lim\limits_{i\to\infty}\tilde u_i(x_k)\in\R\cup\{\infty\}$
  for $k\in\N$. This limit exists by the construction of the
  subsequence $(\tilde u_i)_{i\in\N}$. By passing to the limit in the
  H\"older estimate for $\tilde u_i$, we obtain the claimed H\"older
  estimate with $a+\frac12$ for $u$ and $x=x_k,y=x_l$, $k,l\in\N$. Set
  $u(x):=\lim\limits_{k\to\infty}u(x_k)$ for $x\in B$, $x_k\in D$ and
  $x_k\to x$ as $k\to\infty$. The H\"older estimate ensures that $u$
  is well-defined and fulfills the claimed H\"older estimate with
  $a+1$. Set $\Omega:=\{x\in B\colon u(x)<\infty\}$. There, pointwise
  convergence and local H\"older estimates imply locally uniform
  convergence in $\Omega$.
\end{proof}

\begin{remark}
  \neueZeile
  \begin{enumerate}[(i)]
  \item This result extends to families of locally equicontinuous
    functions.
  \item Notice that the functions $u_i$ in the previous lemma are not
    necessarily finite on all of $B$. Hence the lemma can also be
    applied to functions $u_i$ which are not defined in all of $B$: It
    suffices to set $u_i:=+\infty$ outside its original domain of
    definition.
  \item Observe that the domain $\Omega$ obtained in Lemma \ref{AA
      lem} may be empty. However, for the existence result (Theorem
    \ref{exist thm}), the fact that $\Omega\ne\emptyset$ is ensured by
    the choice of initial condition for the approximating solutions
    and Lemma \ref{hoelder t lem}.
  \end{enumerate}
\end{remark}

\section{Existence}
\label{existence sec}
In this section we will use approximate solutions to prove existence
of a singularity resolving solution to mean curvature flow. 

We start by constructing a nice mollification of $\min\{\cdot,\cdot\}$. 
Choose a smooth monotone approximation $f$ of $\min\{\cdot,0\}$ such
that $f(x)=\min\{x,0\}$ for $|x|>1$ and set $\min_\epsilon\{a,b\}
:=\epsilon f\left(\frac1\epsilon(a-b)\right)+b$.\par We
will set $\min_\epsilon\{u(x),L\}:=L$ at $x$ if $u$ is not defined at
$x$.

\begin{lemma}[Existence of approximating solutions]
  \label{ex approx sol lem}
  Let $A\subset\R^{n+1}$ be an open set. Assume that $u_0\colon
  A\to\R$ is locally Lipschitz continuous and maximal.\par
  Let $L>0$, $R>0$ and $1\ge\epsilon>0$. Then there exists a smooth
  solution $u_{\epsilon,R}^L$ to
  \[
  \begin{cases}
    \dot u=\sqrt{1+|Du|^2}\cdot\divergenz
    \left(\dfrac{Du}{\sqrt{1+|Du|^2}}\right) &\text{in }
    B_R(0)\times[0,\infty),\\
    u=L &\text{on }\partial B_R(0)\times[0,\infty),\\
    u(\cdot,0)=\min_\epsilon\big\{u_{0,\epsilon}, L\big\} &\text{in
    }B_R(0),
  \end{cases}
  \]
  where $u_{0,\epsilon}$ is a standard mollification of $u_0$.  We
  always assume that $R\ge R_0(L,\epsilon)$ is so large that $L+1\le
  u_{0,\epsilon}$ on $\partial B_R(0)$.
\end{lemma}
\begin{proof}
  The initial value problem for $u^L_{\epsilon,R}$ involves smooth
  data which fulfill the compatibility conditions of any order for
  this parabolic problem. Hence we obtain a smooth solution
  $u^L_{\epsilon,R}$ for some positive time interval. According to
  \cite{HuiskenBdry}, this solution exists for all positive times.
\end{proof}
Observe that the approximate solutions of Lemma \ref{ex approx sol
  lem} fulfill Assumption \ref{est ass}
with \[\hat\Omega=\left\{(x,t)\colon u_{\epsilon,R}^L<a\right\}\] and
$0$ there replaced by $a$ for any $a<L$.

\begin{theorem}[Existence]
  \label{exist thm}
  Let $A\subset\R^{n+1}$ be an open set. Assume that
  $u_0\colon A\to\R$ is maximal and locally Lipschitz
  continuous. 
  \par Then there exists $\Omega\subset\R^{n+1}\times[0,\infty)$
  such that $\Omega\cap\left(\R^{n+1}\times\{0\}\right)
  =A\times\{0\}$ and a (classical) singularity resolving
  solution $u\colon\Omega\to\R$ with initial value $u_0$.
\end{theorem}
\begin{proof}
  Consider the approximate solutions $u^L_{\epsilon,R}$ given by Lemma
  \ref{ex approx sol lem}.  The a priori estimates of Theorem \ref{C1
    est thm} and Lemma \ref{hoelder t lem} apply to this situation in
  $\big\{(x,t)\in B_R(0)\times[0,\infty)\colon
  u^L_{\epsilon,R}(x,t)\le L-1\big\}$. According to
  \cite{EckerHuiskenInvent}, we get $u^L_{1/i,i}\to u^L$ as
  $i\to\infty$ and $u^L$ is a solution to mean curvature flow with
  initial condition $\min\{u,L\}$.\par
  Let us derive lower bounds for $u^L$ that will ensure maximality of
  the limit when $L\to\infty$.  As the initial value $u_0$ fulfills
  the maximality condition for every $r>0$ we can find $d=d(r)$ such
  that $B_r((x,L-r-1))$ lies below $\graph\min\{u,L\}$ if $|x|\ge
  d$. Hence $u^L(x,t)\ge L-2$ for $0\le t\le\frac1{n+1}
  r-\frac1{2(n+1)}$ if $|x|\ge d$. Therefore for any $T>0$ there
  exists $d\ge0$ such that $u^L(x,t)\ge L-2$ for $|x|\ge d$ and $0\le
  t\le T$. \par
  The estimates of Theorem \ref{C1 est thm}, Theorem \ref{C2 est thm}
  and Theorem \ref{Ck est thm} survive the limiting process and
  continue to hold for $u^L$: We get locally uniform estimates on
  arbitrary derivatives of $u^L$ in compact subsets of
  $\Omega\cap\left(\R^{n+1}\times(0,\infty)\right)$.  The estimate of
  Lemma \ref{hoelder t lem} also survives the limiting process and we
  get uniform bounds for $\left\Vert u^L\right\Vert_{C^{0,1;0,1/2}}$
  in compact subsets of $\Omega$.\par
  Now we apply Lemma \ref{AA lem} to $u^L$, $L\in\N$, and get a
  solution $(\Omega,u)$ and a subsequence of $u^L$, which we assume to
  be $u^L$ itself, such that $u^L\to u$ locally uniformly in
  $\Omega$.\par
  According to Lemma \ref{maximality lem}, $\Omega$ is open in
  $\R^{n+1}\times[0,\infty)$.  The $C^{0,1;0,1/2}$-estimates imply
  that the domains of definition of $u_0$ and $u|_{t=0}$ coincide. In
  particular in Definition \ref{mcf def} we get $A=\Omega_0(\Omega)$
  and $u(\cdot,0)=u_0$.
  \par
  The derivative estimates and local interpolation inequalities of the
  form \[\Vert Dw\Vert^2_{C^0(B)}\le c(n,B)\cdot\Vert
  w\Vert_{C^0(B)}\cdot \Vert w\Vert_{C^2(B)}\] for any $w\in C^2$ and
  any ball $B$ (see e.\,g.{} \cite[Lemma A.5]{OSFSMShyp}) imply that
  $u^L\to u$ smoothly in
  $\Omega\cap\left(\R^{n+1}\times(0,\infty)\right)$. Hence $u$
  fulfills the differential equation for graphical mean curvature
  flow.\par The lower bound $u^L(x,t)\ge L-2$ above for $|x|\ge d$ and
  Lemma \ref{maximality lem} imply maximality.
  \par Hence, we obtain the existence of a singularity resolving
  solution $(\Omega,u)$ for each maximal Lipschitz continuous function
  $u_0\colon A\to\R$.
\end{proof}
\begin{remark}
Notice that in the proof of Theorem \ref{exist thm} we started with the approximate solutions of Lemma \ref{ex
  approx sol lem} instead of $u^L$ in the proof of Theorem \ref{exist
  thm} as the former are smooth up to $t=0$ and allow to apply our a
priori estimates.  
\end{remark}
\section{The level set flow and singularity resolving solutions}
\label{levelset}

In this section we explore the relation between level set solutions as
defined at the beginning of Appendix \ref{lsf appendix} 
and singularity resolving solutions given by Theorem \ref{exist
  thm}. More precisely, we prove the following result
\begin{theorem} 
  \label{mainlevelsettheo}
  Let $(\Omega,u)$ be a solution to mean curvature flow as in Theorem
  \ref{exist thm}. Let $\partial\D_t$ be the level set evolution of
  $\partial\Omega_0$ as defined below. If $\partial\D_t$ does not
  fatten, the measure theoretic boundaries of $\Omega_t$ and $\D_t$
  coincide for every $t\ge0$:
  $\partial^\mu\Omega_t= \partial^\mu\D_t$.
\end{theorem}
For the definition of a level set solution and fattening, we refer to
Appendix \ref{lsf appendix}.

In order to prove Theorem \ref{mainlevelsettheo} we need a few
definitions which we summarize in Table \ref{notion tab}. Unless
stated otherwise, we will always assume that we consider signed
distance functions which are truncated between $-1$ and $1$, i.\,e.{}
we consider $\max\{-1,\min\{d,1\}\}$, and negative inside the set or
above the graph considered.

\begin{enumerate}[(i)]
\item Let $\tilde v\colon\R^{n+1}\times[0,\infty) \to\R$ be the
  solution to \eqref{weq} such that $\tilde v(\cdot,0)$ is the
  distance function to $\partial\Omega_0$. Set
  $\D_t:=\left\{x\in\R^{n+1}\colon \tilde v(x,t)<0\right\}$.
\item Let $v\colon\R^{n+2}\times[0,\infty)\to\R$ be the solution to
  \eqref{weq} such that $v(\cdot,0)$ is the distance function to
  $\partial\Omega_0\times\R$. Set
  $C_t:=\left\{\left(x,x^{n+2}\right)\in\R^{n+2}\colon
    v\left(x,x^{n+2},t\right)<0\right\}$.
\item Let $w\colon\R^{n+2}\times[0,\infty)\to\R$ be the solution to
  \eqref{weq} such that $w(\cdot,0)$ is the distance function to
  $\graph u(\cdot,0)|_{\Omega_0}$. Set
  $E_t:=\left\{\left(x,x^{n+2}\right) \colon
    w\left(x,x^{n+2},t\right)<0 \right\}$. 
\end{enumerate}

\begin{table}[h]
  \centering
  \begin{tabular}{|c|c|c|}
    \hline
    solution to \eqref{weq} & initial set & set\\\hline
    $w$ & $\graph u_0$ & $E_t$\\\hline
    $\tilde v$ & $\partial\Omega_0$ & $\D_t$\\\hline
    $v$ & $\partial\Omega_0\times\R$ & $C_t$\\\hline
  \end{tabular}
  \medskip
  \caption{Notation for weak solutions}
  \label{notion tab}
\end{table}

Theorem \ref{mainlevelsettheo} will follow from
\begin{proposition} 
  \label{main level set prop}
  Let $(\Omega,u)$ be a solution to mean curvature flow as in Theorem
  \ref{exist thm}. If the level set evolution of $\partial\Omega_0$
  does not fatten, we obtain $\mathcal H^{n+1}$-almost everywhere that
  $\Omega_t=\D_t$ for all $t\ge0$, i.\,e.{} $\mathcal
  H^{n+1}(\Omega_t\triangle\D_t)=0$ for every $t\ge0$.
\end{proposition}

We start by showing that $v$ and $\tilde v$ are closely related.
\begin{lemma}\label{v v lem}
  For $v$ and $\tilde v$ as above, we have $v\left(x,x^{n+2},t\right)
  =\tilde v(x,t)$ for all points
  $\left(x,x^{n+2},t\right)\in\R^{n+1}\times\R\times[0,\infty)$. This
  implies $\D_t\times\R =C_t$ and $\D_t^+\times\R =C^+_t$.
\end{lemma}
\begin{proof}
  This follows directly from uniqueness of solutions to \eqref{weq} as
  $v\left(x,x^{n+2},0\right)=\tilde v(x,0)$. See Theorem
  \ref{existence}.
\end{proof}

\begin{lemma}\label{E C lem}
  We have $w\ge v$. In particular, $E_t^+\subset C^+_t$.
\end{lemma}
\begin{proof}
  This follows from $w(\cdot,0)\ge v(\cdot,0)$ and Theorem
  \ref{adaptedcp}. 
\end{proof}

\begin{lemma}\label{graph E lem}
  We have $\graph u(\cdot,t)\subset \partial E^+_t$. 
\end{lemma}
\begin{proof}
  Let $w^L\colon\R^{n+2}\times[0,\infty)\to\R$ be the solution to
  \eqref{weq} with $w^L(\cdot,0)$ equal to the distance function to
  $\graph u^L$, where $u^L$ is as in Theorem \ref{exist
    thm}. According to \cite{BitonNonfattening} the solution $w^L$
  does not fatten: For each $\epsilon>0$ there is a $\delta>0$ such
  that the inequality $w^L(x,0)\geq w^L\left(x+\epsilon
    e^{n+2},0\right)+\delta$ holds if we truncate at appropriate
  heights.  By Theorem \ref{well-posed} and Theorem \ref{adaptedcp} we
  have that $w^L(x,t)\geq w^L\left(x+\epsilon e^{n+2},t\right)+\delta$
  near the zero level set. Hence
  $\left\{\left(x,x^{n+2}\right)\in\R^{n+1}\times\R\colon
    w^L\left(x,x^{n+2},t\right)=0\right\} =\graph u^L(\cdot,t)$.\par
  Observe that $w^L(\cdot,0)\nearrow w(\cdot,0)$. Hence Theorem
  \ref{monotoneconv} implies that $w^L(\cdot,t)\nearrow w(\cdot,t)$
  for all $t\ge0$. \par Let $x^{n+2}<u(x,t)$. Then $x^{n+2}<u^L(x,t)$
  for some $L$ and hence $w^L\left(x,x^{n+2},t\right)>0$. Since
  $w\left(x,x^{n+2},t\right)\ge w^L\left(x,x^{n+2},t\right)>0$ we have
  that
  \begin{equation}
    \label{u w incl eq}
    \left\{ \left(x, x^{n+2}\right)\colon x^{n+2}<u(x,t)\right\}
    \subset \left\{ \left(x, x^{n+2}\right)\colon 0<w\left(x,
        x^{n+2},t\right)\right\}. 
  \end{equation}
  On the other hand, for every $\left(x,x^{n+2},t\right)$ such that
  $x^{n+2}=u(x,t)$ there is a sequence $\left(x, u^L(x,t)\right)_L$
  such that $\left(x, u^L(x,t)\right)\to \left(x, u(x,t)\right)$ as
  $L\to \infty$.  Moreover, since the $w^L$ converge monotonically,
  the convergence is locally uniform. We conclude that
  $$0 =\lim\limits_{L\to\infty} w^L\left(x, u^L(x,t),t\right)=w
  \left(x, u(x,t),t\right).$$ This concludes the proof of $\graph
  u(\cdot,t)\subset\partial E^+_t$. \par
  By arguments similar to those used for proving \eqref{u w incl eq},
  we can show that
  \[\left\{ \left(x, x^{n+2}\right)\colon x^{n+2}>u(x,t)\right\}
  \subset \left\{ \left(x, x^{n+2}\right)\colon w\left(x,
      x^{n+2},t\right)\leq 0\right\}.  \qedhere\]
\end{proof}

\begin{corollary}
  \label{x outside Omega pos cor}
  Let $x\not\in\Omega_t$ then $w\left(x,x^{n+2},t\right)>0$ for any
  $x^{n+2}$.
\end{corollary}
\begin{proof}
  The above argument in the case $x^{n+2}<u(x,t)$ also extends to the
  case $u(x,t)=+\infty$. 
\end{proof}

\begin{corollary}
  \label{omega Dmu cor}
  If $C_t$ or, equivalently, $\D_t$ does not fatten, then
  $\Omega_t\subset \D^\mu_t$.
\end{corollary}
\begin{proof}
  Combining Lemmata \ref{v v lem}, \ref{E C lem} and \ref{graph E lem}
  we get $\graph u(\cdot,t)\subset\D_t^+\times\R$. This implies
  $\Omega_t\subset\D_t^+$. As $\D_t$ is not fattening, we see that
  $\mathcal H^{n+1}\left(\D_t^+\setminus\D_t\right)=0$. Notice that
  $\D_t\subset \D_t^\mu \subset \D^+_t$.  As $\Omega_t$ is an open
  set, the claim follows.
\end{proof}

The following lemma shows that $\graph u(\cdot,t)$ does not ``detach''
from the evolving cylinder at infinity. 
\begin{lemma}\label{Dt in omegat}
  We have $\D_t\subset \Omega_t$.
\end{lemma}
\begin{proof}
  Denote by $w^R$ the solution to \eqref{weq} with initial condition
  the distance function to the set $\graph(u_0-R)$.
    
  Notice that $w^R(\cdot, 0)\searrow v(\cdot,0) $ as $R\to
  \infty$. Theorem \ref{monotoneconv} implies that
  \begin{equation} \label{convergenceofd} w^R(\cdot,t)\searrow
    v(\cdot,t) \hbox{ as }R\to \infty. \end{equation}
  
  Suppose there are $x, \;t$ such that $x\in \D_t\setminus
  \Omega_t$. Then by Corollary \ref{x outside Omega pos cor} it would
  hold for every $R>0$ and $x^{n+2}$ that
  $$w^R(x, x^{n+2},t)\ge 0 \hbox{ and }v(x , x^{n+2},t)<0. $$
  However, taking $R\to \infty$ this contradicts
  \eqref{convergenceofd}.
\end{proof}

\begin{proof}[Proof of Proposition \ref{main level set prop}]
  According to Corollary \ref{omega Dmu cor} and Lemma \ref{Dt in
    omegat} we have \[\D_t\subset\Omega_t\subset \D_t^\mu \subset
  \D^+_t.\] If there is no fattening $\mathcal
  H^{n+1}\left(\D_t^+\setminus \D_t\right)=0$. The claim follows.
\end{proof}

\begin{remark}
  \label{lsf sing flow rem}
  \neueZeile
  \begin{enumerate}[(i)]
  \item From Proposition \ref{main level set prop} we have that
    \[\sup\,\{t\ge0\colon u(\cdot,t)\not\equiv\infty\} 
    =\sup\left\{t\ge0\colon \D_t\ne\emptyset\right\},\] i.\,e.{} the
    singularity resolving solution vanishes at the same time as the
    level set solution. Here $u(x,t)=\infty$ is to be understood as in
    Lemma \ref{AA lem}.
  \item Generically, level set solutions do not fatten, see
    \cite{TomEllReg}. Examples of initial conditions that do not
    fatten are mean convex hypersurfaces (see \cite{WhiteNature}) and
    star-shaped domains of definition (see \cite{BitonNonfattening}
    and references therein).
  \item Under conditions similar to \cite{BitonNonfattening} it is
    possible to prove that $w$ does not fatten and that $(\Omega,u)$
    is unique. 
  \item Theorem \ref{mainlevelsettheo} also holds if the $\partial
    \Omega_0$ non-fattening assumption is replaced by non-fattening of
    the level set solution with initial condition $ \graph u_0$.
  \item\label{regular item} If $D\tilde v\neq0$ along $\{\tilde
    v\ne0\}$, we have $\D_t^\mu=\D_t$ and hence $\Omega_t=\D_t$.
  \end{enumerate}
\end{remark}

\appendix
\section{Definitions and known results for level set flow}
\label{lsf appendix}

Different approaches have been considered in order to define a weak
solution to mean curvature flow via a level set method (see for
example
\cite{ChenGigaGotolevelset,evanssprucklevelset1,JohnHeadPhD,schulzepowers}).
We define it as follows: Given an initial surface $\partial E_0$, we
define a level set solution to mean curvature flow as the set
$\partial E_t=\partial\{x: w(x,t)<0\}$, where $w$ satisfies in the
viscosity sense the equation
\begin{equation}\label{weq}
  \begin{cases} 
    \frac{\partial w}{\partial
      t}-\left(\delta^{ij}-\frac{w^{i}w^{j}}{|Dw|^2}\right)w_{ij}= 0
    &\hbox{ in }\R^{n+2}\times (0,\infty), \\ w(\cdot,0)= w_0(\cdot)
    &\hbox{ in }\R^{n+2}. 
  \end{cases}
\end{equation} 
and $E_0=\{x:w_0(x)<0\}$. We also set $E^+_t:=\{x\colon
w(x,t)\le0\}$. 

We say that a solution to \eqref{weq} does not fatten if $$\mathcal
H^{n+2}(\{w(\cdot,t)=0\})=0$$ for all $t\ge0$, where $\mathcal
H^{n+2}$ denotes the $(n+2)$-dimensional Hausdorff measure.

Observe that our definition of solutions differs from the notion in
\cite{ChenGigaGotolevelset,evanssprucklevelset1}: They define the
level set solution to be $\{x\colon w(x,t)=0\}$.  If there is
fattening, our definition picks the ``inner boundary''. Often,
however, these definitions coincide, see e.\,g.{}
\cite{evanssprucklevelset3,JohnHeadPhD}.

Let $E\subset\R^{n+2}$ be measurable. We define the open set $E^\mu$,
the measure theoretic interior of $E$,
by \[E^\mu:=\left\{x\in\R^{n+2}\colon \exists\, r>0\colon |B_r(x)|
  =|E\cap B_r(x)|\right\}.\] If $E$ is open, we get $E\subset E^\mu
\subset \ol E$. We also define the measure theoretic boundary
$\partial^\mu E$ of $E$ by \[\partial^\mu E:=\left\{x\in\R^{n+2}\colon
  \forall\, r>0\colon 0<|E\cap B_r(x)|<|B_r(x)|\right\}.\]

In what follows we summarize some results in the literature that will
be used in our proofs.  We will work with the class $BUC(Z)$ which are
functions uniformly continuous and bounded in $Z\subset \R^{n+2}\times
[0,T)$.
    
\begin{theorem}[Existence {\cite[Theorem
    4.3.5]{Gigabook}}] \label{existence} If $w_0 \in
  BUC\left(\R^{n+2}\right)$ then there is a unique viscosity solution
  $w\in BUC\left(\R^{n+2}\times[0,\infty)\right)$ to \eqref{weq}.
\end{theorem}
 
\begin{theorem}[Geometric Uniqueness
  \cite{evanssprucklevelset1,Gigabook}] \label{well-posed} Let
  $w_1(x,t)$ and $w_2(x,t)$ be viscosity solutions to \eqref{weq} such
  that
  $$\{x:w_1(x,0)=0\}=\{x:w_2(x,0)=0\},$$
  then $$\{x:w_1(x,t)=0\}=\{x:w_2(x,t)=0\}$$ for any $t>0$.  
\end{theorem}

Following Theorem 3.1.4 in \cite{Gigabook} we have the following
result for continuous sub- and super-solutions:
\begin{theorem}[Comparison principle]\label{adaptedcp}
  Let $w$ and $v$ be continuous sub- and super-solutions of
  \eqref{weq}, respectively, in the viscosity sense in $\R^{n+2}
  \times [0,T)$.  Assume that $w$ and $-v$ are bounded from above in
  $\R^{n+2}\times[0,T)$. Assume that
  $$w(x,0)-v(x,0)\leq 0$$
  then
  $$w(x,t)-v(x,t)\leq 0 \hbox{ for } (x,t)\in  \R^{n+2}\times[0,T).$$
\end{theorem}
 
\begin{theorem}[Monotone Convergence {\cite[Lemma
    4.2.11]{Gigabook}}] \label{monotoneconv} Consider functions
  \\$w_{0,m}, \; w_0 \in BUC(\R^n)$ such that $w_{0,m}\nearrow
  w_0$. Then if $w_m$ and $w$ are solutions to \eqref{weq} with
  initial data $w_{0,m}$ and $w_0$, respectively, we have for every
  time that $w_{m}\nearrow w$.
\end{theorem}

\begin{remark}
  \neueZeile
  \begin{enumerate}[(i)]
  \item The (non-truncated) signed distance function to $\partial E$
    may be defined as $d_E(x)=\dist(x, E)- \dist\left(x, \R^m\setminus
      E\right)$. In particular, we assume that the
    signed distance function to $\partial E$ is negative for every
    $x\in E$.
  \item In general, the initial conditions considered in Section
    \ref{levelset} will be given by truncated distance function to a
    set.
  \item If the set $\partial \Omega_0$ is compact and evolves smoothly
    under mean curvature flow, the level set formulation above agrees
    with the classical solution.
\end{enumerate}
\end{remark}

\def\emph#1{\textit{#1}}
\bibliographystyle{amsplain}
\def\us{\underline{\phantom{x}}}

\end{document}